\documentclass[a4paper, 10pt, oneside, onecolumn]{article}

\usepackage[utf8]{inputenc}
\usepackage[T1]{fontenc}
\usepackage{lmodern}
\usepackage{amsmath}
\usepackage{amssymb}
\usepackage{amsfonts}
\usepackage{amsthm}
\usepackage{dsfont}
\usepackage{mathtools}
\usepackage{marginnote}
\usepackage{xcolor}
\usepackage{hyperref}
\hypersetup{
  colorlinks=false,
  pdfborder={0 0 0},
  pdftitle={Boundedness enforced by mildly saturated conversion in a chemotaxis-May-Nowak model for virus infection},
  pdfauthor={Mario Fuest},
  pdfkeywords={boundedness, chemotaxis, May-Nowak model, virus dynamics},
  bookmarksopen=true,
}

\usepackage[backend=bibtex, giveninits=true]{biblatex}
\AtEveryBibitem{\clearfield{issn}}
\bibliography{\jobname}

\RequirePackage{geometry}
\newcommand{\R}{\mathbb{R}}

\newcommand{\N}{\mathbb{N}}

\newcommand{\ur}[1]{\mathrm{#1}}
\newcommand{\ure}{\ur e}

\ifdefined\labelenumi
  \renewcommand{\labelenumi}{(\roman{enumi})}
  
\fi

\newcommand{\gt}{>}
\newcommand{\lt}{<}

\newcommand{\ra}{\rightarrow}

\newcommand{\nea}{\nearrow}

\newcommand{\ol}{\overline}

\newcommand{\ds}{\,\mathrm{d}s}

\newcommand{\ddt}{\frac{\mathrm{d}}{\mathrm{d}t}}

\newcommand{\hp}{\hphantom}
\newcommand{\pe}{\mathrel{\hp{=}}}

\newcommand{\tmax}{T_{\max}}
\newcommand{\intom}{\int_\Omega}
\newcommand{\ombar}{\ol \Omega}

\let\originalparagraph\paragraph
\renewcommand{\paragraph}[2][.]{\originalparagraph{#2#1}}

\newtheorem{base}{Base}[section]
\numberwithin{equation}{section}

\theoremstyle{plain}
\newtheorem{theorem}[base]{Theorem}
\newtheorem{lemma}[base]{Lemma}

\theoremstyle{definition}

\newtheorem{remark}[base]{Remark}

\begin{document}
\title{Boundedness enforced by mildly saturated conversion in a chemotaxis-May-Nowak model for virus infection}
\author{
Mario Fuest\footnote{fuestm@math.uni-paderborn.de}\\
{\small Institut f\"ur Mathematik, Universit\"at Paderborn,}\\
{\small 33098 Paderborn, Germany}
}
\date{}

\maketitle

\begin{abstract}
  \noindent
  We study the system
  \begin{align*} \label{prob:star} \tag{$\star$}
    \begin{cases}
      u_t = \Delta u - \nabla \cdot (u \nabla v) - u - f(u) w + \kappa, \\
      v_t = \Delta v - v + f(u) w, \\
      w_t = \Delta w - w + v,
    \end{cases}
  \end{align*}
  which models the virus dynamics in an early stage of an HIV infection,
  in a smooth, bounded domain $\Omega \subset \R^n, n \in \N,$ for a parameter $\kappa \ge 0$
  and a given function $f \in C^1([0, \infty))$ satisfying $f \ge 0$, $f(0) = 0$
  and $f(s) \le K_f s^\alpha$ for all $s \ge 1$, some $K_f \gt 0$ and $\alpha \in \R$.\\[5pt]
  We prove that whenever
  \begin{align*}
    \alpha \lt \frac2n,
  \end{align*}
  solutions to \eqref{prob:star} exist globally and are bounded.
  The proof mainly relies on smoothing estimates for the Neumann heat semigroup
  and (in the case $\alpha \gt 1$) on a functional inequality.\\[5pt]
  Furthermore, we provide some indication why the exponent $\frac2n$ could be essentially optimal.\\[5pt]
  \textbf{AMS Classification (2010):} {35A01, 35B45, 35K57, 35Q92, 92C17}
\end{abstract}

\section{Introduction}
\paragraph{The model} The system
\begin{align} \label{prob:may_nowak}
  \begin{cases}
    u' = - u - uw + \kappa, \\
    v' = - v + uw, \\
    w' = - w + v.
  \end{cases}
\end{align}
can be used to describe the dynamics of a virus infection.
Here $u$ and $v$ model healthy and infected cells, respectively,
and $w$ denotes the concentration of virus particles.
All three populations undergo spontaneous decay, healthy cells are produced by a fixed rate $\kappa \ge 0$,
on contact with virus particles healthy cells are converted to infected ones and the virus is produced by infected cells.

This system (or similar ones including more parameters) has been motivated
and analyzed from a biological
\cite{BonhoefferEtAlVirusDynamicsDrug1997,
NowakBanghamPopulationDynamicsImmune1996,
NowakMayVirusDynamicsMathematical2000,
PerelsonEtAlHIV1DynamicsVivo1996,
WeiEtAlViralDynamicsHuman1995}
as well as a mathematical point of view \cite{CamposEtAlEffectsDistributedLife2008, JonesRoemerAnalysisSimulationThreecomponent2014}.
Questions of large time behavior are essentially answered,
where (in absence of other parameters) the sign of $\kappa - 1$
has been detected to play an important role \cite{KorobeinikovGlobalPropertiesBasic2004}.

This model is therefore indeed able to give insights into virus dynamics;
however, an ODE system can by its nature never capture nontrivial spatial effects such as pattern formation.
Hence, in \cite{StancevicEtAlTuringPatternsDynamics2013} the model
\begin{align} \label{prob:p_orig}
  \begin{cases}
    u_t = \Delta u - \nabla \cdot (u \nabla v) - u - uw + \kappa, \\
    v_t = \Delta v - v + uw, \\
    w_t = \Delta w - w + v
  \end{cases}
\end{align}
has been suggested.
In addition to \eqref{prob:may_nowak} the populations are now assumed to move around randomly; that is, they diffuse.
Furthermore, healthy cells are attracted by high concentrations of infected cells (which is a realistic assumption for HIV infections),
modeled by the cross-diffusion term $-\nabla \cdot (u \nabla v)$.
Systems containing such a term have been introduced by Keller and Segel in their seminal work \cite{KellerSegelTravelingBandsChemotactic1971} in order to model the behavior of E.\ coli bacteria
and systems including cross-diffusion have been of great interest both to biologists and mathematicians in the past decades,
see for instance \cite{BellomoEtAlMathematicalTheoryKeller2015} for a survey.

In \cite[Section 8]{StancevicEtAlTuringPatternsDynamics2013} it has been argued
that finite-time blow-up is not realistic in this setting.
However, such a phenomenon has been proven to occur for the classical Keller--Segel model.
While solutions are global and bounded if $n = 1$ \cite{OsakiYagiFiniteDimensionalAttractor2001},
blow-up does occur for some initial data in the spatially two-dimensional case 
\cite{HorstmannWangBlowupChemotaxisModel2001, SenbaSuzukiParabolicSystemChemotaxis2001}
and for many initial data in the higher dimensional radial-symmetric setting \cite{WinklerFinitetimeBlowupHigherdimensional2013}.
Therefore one should slightly alter the system in order to prevent blow-up.
One possible modification is to weaken (or in general to alter) the cross-diffusion term;
that is, to replace the first equation in \eqref{prob:p_orig} with
\begin{align*}
  u_t = \Delta u - \nabla \cdot (uf(u) \nabla v) - u - uw + \kappa.
\end{align*}
Here the cases $f \equiv \chi \in \R$, $|\chi|$ small \cite{BellomoEtAlChemotaxisModelVirus2018}
or $|f(s)| \le K_f (1 + s)^{-\alpha}$ for all $s \ge 0$, a constant $K_f \gt 0$ and $\alpha \in \R$ sufficiently large
\cite{HuLankeitBoundednessSolutionsVirus2018, LiTaoBoundednessEnforcedMild2018, WinklerBoundednessChemotaxisMayNowakModel2018}
have been analyzed.

In the present work we alter the conversion term.
To be precise, we study
\begin{align*} \label{prob:p} \tag{\text{P}}
  \begin{cases}
    u_t = \Delta u - \nabla \cdot (u \nabla v) - u - f(u) w + \kappa,  & \text{in $\Omega \times (0, T)$}, \\
    v_t = \Delta v - v + f(u) w,                                       & \text{in $\Omega \times (0, T)$}, \\
    w_t = \Delta w - w + v,                                            & \text{in $\Omega \times (0, T)$}, \\
    \partial_\nu u = \partial_\nu v = \partial_\nu w = 0,              & \text{on $\partial \Omega \times (0, T)$}, \\
    u(\cdot, 0) = u_0, v(\cdot, 0) = v_0, w(\cdot, 0) = w_0,           & \text{in $\Omega$}
  \end{cases}
\end{align*}
for $T \in (0, \infty]$, a smooth, bounded domain $\Omega \subset \R^n, n \in \N$, a parameter $\kappa \ge 0$,
given sufficiently smooth initial data $u_0, v_0, w_0$ and a given function $f \in C^1([0, \infty))$, which satisfies
\begin{align} \label{eq:main:f}
  f \ge 0, \quad f(0) = 0
  \quad \text{and} \quad
  f(s) \le K_f s^\alpha
  \text{ for all $s \ge 1$}
\end{align}
for some $K_f \gt 0$ and $\alpha \in \R$.
For instance, $f: [0, \infty) \ra \R, s \mapsto \frac{s}{1+s^{1-\alpha}}, K_f = 1$
and $f: [0, \infty) \ra \R$, $s \mapsto s^\alpha, K_f = 1$ may serve as prototypes
for the cases $\alpha \le 1$ and $\alpha \gt 1$, respectively.

The function $f$ can be seen as an interpolation between the classical Lotka-Volterra term ($f = \mathrm{id}$, that is $\alpha = 1$)
and other conversion terms, such as Holling's Type II and III or Beddington-DeAngelis responses (all $\alpha = 0$);
see for instance \cite{SkalskiGilliamFunctionalResponsesPredator2001} for definitions and a comparison of functional responses.

The case $\alpha = 0$ has been partly studied (among other modifications of \eqref{prob:p_orig})
in \cite{BellomoTaoStabilizationChemotaxisModel2017}.
The authors proved global existence and showed for the case $\kappa \le 1$ convergence towards stationary steady states.

A natural question is whether similar results can be obtained for $\alpha \gt 0$.
We show that this is indeed the case, at least regarding global existence and boundedness;
our main result is as follows:
\begin{theorem} \label{th:main}
  Let $\Omega \subset \R^n, n \in \N,$ be a smooth bounded domain and $\kappa \ge 0$.
  Suppose $f \in C^1([0, \infty))$ satisfies \eqref{eq:main:f} for some $K_f \gt 0$ and
  \begin{align} \label{eq:main:alpha}
    \alpha \lt \frac{2}{n}.
  \end{align}
  Then for all nonnegative initial data
  \begin{align} \label{eq:main:initial_data}
    u_0 \in C^0(\ombar), \quad
    v_0 \in W^{1, \infty}(\Omega)
    \quad \text{and} \quad
    w_0 \in C^0(\ombar)
  \end{align}
  the system \eqref{prob:p} possesses a nonnegative global classical solution $(u, v, w)$ satisfying
  \begin{align*}
    u &\in C^0(\ombar \times [0, \infty)) \cap C^{2, 1}(\ombar \times (0, \infty)), \\
    v &\in \bigcap_{q \gt n} C^0([0, \infty); W^{1, q}(\Omega)) \cap C^{2, 1}(\ombar \times (0, \infty)) \quad \text{and} \\
    w &\in C^0(\ombar \times [0, \infty)) \cap C^{2, 1}(\ombar \times (0, \infty))
  \end{align*}
  and being uniquely determined by these inclusions.
  
  Additionally, the solution is bounded:
  There exists $C \gt 0$ such that
  \begin{align*}
      \|u(\cdot, t)\|_{L^\infty(\Omega)}
    + \|v(\cdot, t)\|_{W^{1, \infty}(\Omega)}
    + \|w(\cdot, t)\|_{L^{\infty}(\Omega)}
    \lt C
    \quad \text{for all $t \gt 0$}.
  \end{align*}
\end{theorem}

\begin{remark}
  The absence of parameters in \eqref{prob:p} is purely to simplify the notation and to maintain readability.
  Using the same methods as below one could achieve the same result for the system
  \begin{align*}
    \begin{cases}
      u_t = D_1 \Delta u - \chi \nabla \cdot (u \nabla v) - d_1 u - f(u) w + \kappa, \\
      v_t = D_2 \Delta v - d_2 v + f(u) w, \\
      w_t = D_3 \Delta w - d_3 w + r v
    \end{cases}
  \end{align*}
  (with corresponding initial and boundary conditions),
  wherein the parameters $D_1, D_2, D_3, \chi, d_1, d_2, d_3, r \gt 0$ and $\kappa \ge 0$ are given. 
  In \cite[Section 9]{StancevicEtAlTuringPatternsDynamics2013} realistic values of these parameters are discussed.
\end{remark}

\paragraph{Optimality of the exponent}
A natural question is whether the condition~\eqref{eq:main:alpha} is optimal.
Comparing \eqref{prob:p} with
\begin{align} \label{prob:ks_parab_parab}
  \begin{cases}
    u_t = \Delta u - \nabla \cdot (u \nabla v), \\
    v_t = \Delta v - v + f(u)
  \end{cases}
\end{align}
and related systems indicates that this might indeed be the case:

The production term in the second equation in \eqref{prob:ks_parab_parab} is ``only'' $f(u)$
instead of the possibly larger $f(u) w$ in \eqref{prob:p},
hence one might expect that if finite-time blow-up is possible for \eqref{prob:ks_parab_parab}, then also for \eqref{prob:p}.

For \eqref{prob:ks_parab_parab} an analogous result as stated in Theorem~\ref{th:main} has been proved \cite{LiuTaoBoundednessChemotaxisSystem2016}.
On the other hand, if $n = 2$ and $f = \textrm{id}$ then blow-up in finite time may occur
\cite{HorstmannWangBlowupChemotaxisModel2001, SenbaSuzukiParabolicSystemChemotaxis2001}.
Although the question whether blow-up solutions to \eqref{prob:ks_parab_parab} exist for $n \ge 3$ and $\alpha \ge \frac{2}{n}$
is, to our knowledge, open,
at the very least the exponent $\frac{2}{n}$ is critical for a parabolic-elliptic variant of \eqref{prob:ks_parab_parab}:
In \cite{WinklerCriticalBlowupExponent2018} the system
\begin{align} \label{prob:ks_parab_ell_nonlin_prod}
  \begin{cases}
    u_t = \Delta u - \nabla \cdot (u \nabla v), \\
    0   = \Delta v - \frac{1}{|\Omega|} \intom f(u) + f(u),
  \end{cases}
\end{align}
is considered and the following is proved:
\begin{itemize}
  \item
    If $f \in C^1([0, \infty)$ satisfies $f \ge 0, f' \ge 0$,
    $f(s) \le K_f s^\alpha$ for all $s \ge 1$, some $K_f \gt 0$ and $\alpha \lt \frac{2}{n}$,
    solutions to \eqref{prob:ks_parab_ell_nonlin_prod} exist globally and are bounded,
    provided the initial datum $u_0$ is nonnegative and sufficiently regular.

  \item
    If $f \in C^1([0, \infty)$ satisfies $f \ge 0, f' \ge 0$,
    $f(s) \ge K_f s^\alpha$ for all $s \ge 1$, some $K_f \gt 0$ and $\alpha \gt \frac{2}{n}$ and $\Omega$ is a ball,
    then there exists $u_0$ (with arbitrary small mass)
    such that the solution to \eqref{prob:ks_parab_ell_nonlin_prod} blows up in finite time.
\end{itemize}

\paragraph{Plan of the paper}
After this introductory section we provide a statement on local existence of solutions to \eqref{prob:p},
along with an extensibility criterion (Lemma~\ref{lm:ex_crit})
and some basic properties of the solutions, such as nonnegativity and boundedness in $L^1$ (Lemma~\ref{lm:l1_bdd}).

If $\alpha \le 1$, known $L^p$-$L^q$~estimates
provide an $L^\infty$ bound for $w$ and an $L^q$ bound for $\nabla v$ for some $q \gt n$ (Lemma~\ref{lm:bdd_nabla_v_w}).
The main ingredient here is boundedness of $f(u)$ in $L^{1/\alpha_+}(\Omega)$,
provided by condition~\eqref{eq:main:f} and an $L^1$ bound for $u$.
Perhaps somewhat surprisingly and in contrast to for instance \cite{HuLankeitBoundednessSolutionsVirus2018}
and \cite{WinklerBoundednessChemotaxisMayNowakModel2018},
we do not need to cycle this argument
-- a single application per solution component is sufficient.

However, only relying on $L^p$-$L^q$ estimates does not seem to be fruitful for the case $\alpha \gt 1$.
In that case, which in view of \eqref{eq:main:alpha} only concerns us if $n = 1$,
in Lemma~\ref{lm:n1_bdd_vx} we prove boundedness of the functional $\frac{1}{\alpha} \intom u^\alpha + \frac12 \intom v_x^2$ instead.

Having in both cases obtained an $L^q$~bound for $\nabla v$ for some $q \gt n$
we are able to provide an $L^\infty$~bound for $u$ in Lemma~\ref{lm:bdd_u_l_infty},
again making use of $L^p$-$L^q$~estimates.
Combining these bounds finally allows us to conclude global existence of the solutions.

\section{Preliminaries}
We henceforth fix a smooth, bounded domain $\Omega \subset \R^n$, $n \in \N$, $\kappa \ge 0$
as well as $K_f \gt 0$, $\alpha \in \R$ and a function $f \in C^1([0, \infty))$ satisfying \eqref{eq:main:f}.
Furthermore, we also fix nonnegative initial data $u_0, v_0, w_0: \ombar \ra \R$ fulfilling \eqref{eq:main:initial_data}.

We begin by stating a result on local existence of solutions:
\begin{lemma} \label{lm:ex_crit}
  There exist $\tmax \in (0, \infty]$ and uniquely determined functions
  \begin{align*}
    u &\in C^0(\ombar \times [0, \infty)) \cap C^{2, 1}(\ombar \times (0, \tmax)), \\
    v &\in C^0([0, \infty); W^{1, \infty}(\Omega)) \cap C^{2, 1}(\ombar \times (0, \tmax)) \quad \text{and} \\
    w &\in C^0(\ombar \times [0, \infty)) \cap C^{2, 1}(\ombar \times (0, \tmax)),
  \end{align*}
  solving \eqref{prob:p} classically, and are such that if $\tmax \lt \infty$
  \begin{align*}
    \limsup_{t \nea \tmax}
        \|u(\cdot, t)\|_{L^\infty(\Omega)}
      + \|v(\cdot, t)\|_{W^{1, q}(\Omega)}
      + \|w(\cdot, t)\|_{L^\infty(\Omega)}
    = \infty
  \end{align*}
  holds for all $q \gt n$.
  Additionally, these functions are nonnegative.
\end{lemma}
\begin{proof}
  Local existence and the extensibility criterion can be be proved by standard arguments,
  applied for instance in \cite[Theorem 3.1]{HorstmannWinklerBoundednessVsBlowup2005}.
  The main idea is to use Banach's fixed point argument to construct mild solutions
  and then show that these are in fact classical solutions.

  Nonnegativity of $u, v, w$ is a consequence of the maximum principle and the fact that $u_0, v_0, w_0 \ge 0$ and $\kappa \ge 0$.
\end{proof}

In the sequel we will always denote this solution of \eqref{prob:p} by $(u, v, w)$.

\begin{lemma} \label{lm:l1_bdd}
  There exists $C \gt 0$ such that
  \begin{align*}
    \|u(\cdot, t)\|_{L^1(\Omega)} + \|v(\cdot, t)\|_{L^1(\Omega)} \le C \quad \text{for all $t \in (0, \tmax)$}.
  \end{align*}
\end{lemma}
\begin{proof}
  The function $z(t) := \intom u(\cdot, t) + \intom v(\cdot, t)$, $t \in (0, \tmax)$, fulfills
  \begin{align*}
        z'
    =   - \intom u - \intom \frac{uw}{(1+u)^\beta} + \intom \kappa
        - \intom v + \intom \frac{uw}{(1+u)^\beta}
    =   -z + \kappa |\Omega|
  \end{align*}
  in $(0, \tmax)$,
  hence $z \le \max\{\intom u_0, \kappa |\Omega|\}$ in $(0, \tmax)$ by an ODE comparison argument.
  As $u, v \ge 0$ by Lemma~\ref{lm:ex_crit}, this already implies the statement.
\end{proof}

\section{$L^q$ bound for $\nabla v$ and $L^\infty$ bound for $w$}
As it will turn out, we need to distinguish the cases $\alpha \le 1$ and $\alpha \gt 1$ in order to prove bounds for $\nabla v$ and $w$.

\subsection{Case $\alpha \le 1$}
In order to prove boundedness we first establish relationships
between $L^\infty$ bounds for $w$ and $L^q$ bounds for $\nabla v$ (for appropriate $q \ge 1$).

\begin{lemma} \label{lm:bdd_w_l_infty}
  For all $q \gt \max\{1, \frac{n}{2}\}$ there exist $C \gt 0$ and $b \in (0, 1)$ such that
  \begin{align*}
        \|w(\cdot, t)\|_{L^\infty(\Omega)}
    \le C \left( 1 + \sup_{s \in (0, t)} \|\nabla v(\cdot, s)\|_{L^q(\Omega)}^b \right)
  \end{align*}
  holds for all $t \in (0, T_{\max})$.
\end{lemma}
\begin{proof}
  As
  \begin{align*}
    b := \frac{n - \frac{n}{q}}{n + 1 - \frac{n}{q}} \in (0, 1),
  \end{align*}
  we may invoke the Gagliardo-Nirenberg inequality to obtain $c_1 \gt 0$ such that
  \begin{align*}
        \|\varphi\|_{L^q(\Omega)}
    \le c_1 \|\nabla \varphi\|_{L^q(\Omega)}^b \|\varphi\|_{L^1(\Omega)}^{1-b} + c_1 \|\varphi\|_{L^1(\Omega)}
    \quad \text{for all $\varphi \in W^{1, q}(\Omega)$}.
  \end{align*}

  Moreover, known smoothing estimates for the Neumann Laplace semigroup (cf.\ \cite[Lemma~1.3~(i)]{WinklerAggregationVsGlobal2010})
  provide $c_2 \gt 0$ such that
  \begin{align*}
        \|\ure^{\sigma \Delta} \varphi\|_{L^\infty(\Omega)}
    \le c_2 \left(1 + \sigma^{-\frac{n}{2q}} \right) \|\varphi\|_{L^q(\Omega)}
    \quad \text{for all $\sigma \gt 0$ and all $\varphi \in L^q(\Omega)$}.
  \end{align*}

  As there also exists $c_3 \gt 0$ such that $\|v(\cdot, t)\|_{L^1(\Omega)} \le c_3$
  for all $t \in (0, \tmax)$ by Lemma~\ref{lm:l1_bdd},
  a combination of these estimates and an application of the maximum principle yields
  \begin{align*}
    &\pe  \|w(\cdot, t)\|_{L^\infty(\Omega)} \\
    &\le  \|\ure^{t (\Delta - 1)} w_0\|_{L^\infty(\Omega)}
          + \int_0^t \|\ure^{(t-s) (\Delta -1)} v(\cdot, s)\|_{L^\infty(\Omega)} \ds \\
    &\le  \ure^{-t} \|w_0\|_{L^\infty(\Omega)}
          + c_2 \int_0^t \left( 1 + (t-s)^{-\frac{n}{2q}} \right) \ure^{-(t-s)} \|v(\cdot, s)\|_{L^q(\Omega)} \ds \\
    &\le  \|w_0\|_{L^\infty(\Omega)}
          + c_1 c_2 \int_0^t \left( 1 + (t-s)^{-\frac{n}{2q}} \right) \ure^{-(t-s)}
            \left( \|\nabla v(\cdot, s)\|_{L^q(\Omega)}^b \|v(\cdot, s)\|_{L^1(\Omega)}^{1-b} + \|v(\cdot, s)\|_{L^1(\Omega)} \right) \ds \\
    &\le  \|w_0\|_{L^\infty(\Omega)}
          + c_1 c_2 c_4
            \left( \sup_{s \in (0, t)} \|\nabla v(\cdot, s)\|_{L^q(\Omega)}^a c_3^{1-b} + c_3 \right)
  \end{align*}
  for all $t \in (0, \tmax)$, where
  \begin{align*}
    c_4 := \int_0^\infty \left( 1 + s^{-\frac{n}{2q}} \right) \ure^{-s} \ds
  \end{align*}
  is finite because of $q \gt \frac{n}{2}$.
\end{proof}
 
\begin{lemma} \label{lm:bdd_nabla_v_l_q}
  Suppose $\alpha \le 1$ and $\frac{n}{(\alpha n - 1)_+} \gt 2$.
  Let
  \begin{align} \label{eq:bdd_nabla_v_l_q:ass_q}
    q \in \left[2, \frac{n}{(\alpha n - 1)_+} \right).
  \end{align}
  Then we can find $C \gt 0$ such that
  \begin{align*}
        \|\nabla v(\cdot, t)\|_{L^q(\Omega)}
    \le C \left(1 + \sup_{s \in (0, t)} \|w(\cdot, s)\|_{L^\infty(\Omega)} \right)
  \end{align*}
  holds for all $t \in (0, T_{\max})$.
\end{lemma}
\begin{proof}

  Let $\lambda := \min\{\frac{1}{\alpha_+}, q\} \ge 1$.
  By standard smoothing estimates for the Neumann Laplace semigroup (cf.\ \cite[Lemma~1.3~(ii) and (iii)]{WinklerAggregationVsGlobal2010}; note that $q \ge 2$)
  there exist $c_1, c_2 \gt 0$ such that we have
  \begin{align*}
        \|\nabla \ure^{\sigma \Delta} \varphi\|_{L^q(\Omega)}
    \le c_1 \|\varphi\|_{W^{1, q}(\Omega)}
    \quad \text{for all $\sigma \gt 0$ and all $\varphi \in W^{1, q}(\Omega)$}
  \end{align*}
  as well as
  \begin{align*}
        \|\nabla \ure^{\sigma \Delta} \varphi\|_{L^q(\Omega)}
    \le c_2 \left(1 + \sigma^{-\frac12 - \frac{n}{2}(\frac{1}{\lambda} - \frac{1}{q})}\right) \|\varphi\|_{L^\lambda(\Omega)}
    \quad \text{for all $\sigma \gt 0$ and all $\varphi \in L^\lambda(\Omega)$}.
  \end{align*}
  
  Therefore,
  \begin{align*}
    &\pe  \|\nabla v(\cdot, t)\|_{L^q(\Omega)} \\
    &\le  \| \nabla \ure^{t(\Delta - 1)} v_0 \|_{L^q(\Omega)}
          + \int_0^t \left\| \nabla \ure^{(t-s) (\Delta - 1)} f(u(\cdot, s)) w(\cdot, s) \right\|_{L^q(\Omega)} \ds \\
    &\le  c_1 \ure^{-t} \|v_0\|_{W^{1, q}(\Omega)}
          + c_2 K_f \sup_{s \in (0, t)} \left\|f(u(\cdot, s)) w(\cdot, s)\right\|_{L^\lambda(\Omega)}
            \int_0^t \left(1 + (t-s)^{-\frac12 - \frac{n}{2}(\frac{1}{\lambda} - \frac{1}{q})}\right) \ure^{-(t-s)} \ds
  \end{align*}
  for all $t \in (0, \tmax)$.

  For $s \ge 1$ and $\alpha \le 0$ we have $s^\alpha \le 1$,
  hence the condition \eqref{eq:main:f} implies that independently of the sign of $\alpha$ we have
  \begin{align*}
    f(s) \le K_f s^{\alpha_+} \quad \text{for all $s \ge 1$,}
  \end{align*}
  hence by Lemma~\ref{lm:l1_bdd} we have for some $c_3 \gt 0$
  \begin{align*}
    &\pe  \sup_{s \in (0, t)} \left\|f(u(\cdot, s)) w(\cdot, s)\right\|_{L^\lambda(\Omega)} \\
    &\le  \sup_{s \in (0, t)} \left(
            \left( \intom |f(u(\cdot, s)) w(\cdot, s)|^\lambda \mathds 1_{\{u \ge 1\}} \right)^\frac1\lambda
            + \left( \intom |f(u(\cdot, s)) w(\cdot, s)|^\lambda \mathds 1_{\{u \lt 1\}} \right)^\frac1\lambda
          \right) \\
    &\le  \sup_{s \in (0, t)} \left(
            K_f \|u^{\alpha_+}(\cdot, s)\|_{L^\lambda(\Omega)} \|w(\cdot, s)\|_{L^\infty(\Omega)}
            + \|f\|_{C^0([0, 1])} |\Omega|^\frac1\lambda \|w(\cdot, s)\|_{L^\infty(\Omega)}
          \right) \\
    &\le  \sup_{s \in (0, t)} \left(
            \left( K_f |\Omega|^{1 - \lambda \alpha_+} \|u(\cdot, s)\|_{L^1(\Omega)}^{\alpha_+}
            + \|f\|_{C^0([0, 1])} |\Omega|^\frac1\lambda \right) \|w(\cdot, s)\|_{L^\infty(\Omega)}
            \right) \\
    &\le  c_3 \sup_{s \in (0, t)} \|w(\cdot, s)\|_{L^\infty(\Omega)}
  \end{align*}
  for all $t \in (0, \tmax)$.

  By \eqref{eq:bdd_nabla_v_l_q:ass_q} we have in the case $\alpha \gt 0$ and $\lambda = \frac{1}{\alpha}$
  \begin{align*}
        -\frac12 - \frac{n}{2} \left(\frac{1}{\lambda} - \frac{1}{q}\right)
    \gt -\frac12 - \frac{n}{2} \left(\alpha - \frac{(\alpha n - 1)_+}{n} \right)
    \ge -\frac12 - \frac{n}{2n} 
    =   -1,
  \end{align*}
  whereas, if $\alpha \le 0$ or $\lambda = q$ (and hence in both cases $\lambda = q$)
  \begin{align*}
        -\frac12 - \frac{n}{2} \left(\frac{1}{\lambda} - \frac{1}{q}\right)
    =   -\frac12 - \frac{n}{2} \left(\frac{1}{q} - \frac{1}{q}\right)
    \gt -1
  \end{align*}
  holds.
  In both cases we conclude
  \begin{align*}
    c_4 := \int_0^\infty \left(1 + s^{-\frac12 - \frac{n}{2}(\frac{1}{\lambda} - \frac{1}{q})}\right) \ure^{-s} \ds \lt \infty,
  \end{align*} 
  so that the statement follows for $C := c_1 \|v_0\|_{W^{1, q}(\Omega)} + c_2 c_3 c_4 K_f$.
\end{proof}

We now combine the previous lemmata to obtain
\begin{lemma} \label{lm:bdd_nabla_v_w}
  If $\alpha \le 1$ and $\alpha \lt \frac{2}{n}$,
  then there exist $q \gt n$ and $C \gt 0$ such that
  \begin{align} \label{eq:bdd_nabla_v_w:statement}
    \|\nabla v(\cdot, t)\|_{L^q} + \|w(\cdot, t)\|_{L^\infty(\Omega)} \le C \quad \text{in $(0, \tmax)$}.
  \end{align}
\end{lemma}
\begin{proof}
  By assumption on $\alpha$ we have
  \begin{align*}
        \frac{n}{(\alpha n - 1)_+}
    \gt \frac{n}{2 - 1}
    =   n.
  \end{align*}
  As also $\frac{1}{(\alpha-1)_+} = \infty$,
  we can then find $q \in [2, \frac{n}{(\alpha n - 1)_+}) \cap (n, \infty)$.

  Lemma~\ref{lm:bdd_w_l_infty} and Lemma~\ref{lm:bdd_nabla_v_l_q} provide $C_1, C_2 \gt 0$ and $b \in (0, 1)$
  such that
  \begin{align} \label{eq:bdd_nabla_v_w:nabla_v}
    \|\nabla v(\cdot, t)\|_{L^q(\Omega)}  &\le C_1 \left(1 + \sup_{s \in (0, t)} \|w(\cdot, s)\|_{L^\infty(\Omega)} \right)
  \intertext{and} \notag
    \|w(\cdot, t)\|_{L^\infty(\Omega)}    &\le C_2 \left(1 + \sup_{s \in (0, t)} \|\nabla v(\cdot, s)\|_{L^q(\Omega)}^b \right)
  \end{align}
  hold for all $t \in (0, \tmax)$,
  hence $N(t) := \sup_{s \in (0, t)} \|w(\cdot, s)\|_{L^\infty(\Omega)}$ fulfills
  \begin{align*}
    N(t) \le C_2 + C_1^a C_2 (1 + N(t))^b.
  \end{align*}
  for all $t \in (0, \tmax)$.
  Let $t \in (0, \tmax)$. If $N(t) \ge 1$, then
  \begin{align*}
        N^{1-b}(t)
    \le C_2 N^{-b}(t) + C_1^a C_2 (N^{-b}(t) + 1)^b
    \le C_2 + 2^b C_1^b C_2 
    =:  C_3.
  \end{align*}
  Therefore, for all $t \in (0, \tmax)$, we have $N(t) \le \max\{1, C_3^\frac{1}{1-b}\}$.
  Together with \eqref{eq:bdd_nabla_v_w:nabla_v} this implies the existence of $C \gt 0$
  such that \eqref{eq:bdd_nabla_v_w:statement} holds.
\end{proof}

\subsection{Case $\alpha \gt 1$}
The conditions $\alpha \lt \frac{2}{n}$ and $\alpha \gt 1$ can only be simultaneously fulfilled
if $n = 1$. In that case we are able to immediately derive an $L^\infty$ bound for $w$:

\begin{lemma} \label{lm:n1_bdd_w}
  If $n = 1$, then $\{w(\cdot, t): t \in (0, \tmax)\}$ is bounded in $L^\infty(\Omega)$.
\end{lemma}
\begin{proof}
  By the maximum principle
  and known smoothing estimates for the Neumann Laplace semigroup (cf.\ \cite[Lemma~1.3~(i)]{WinklerAggregationVsGlobal2010})
  we have for some $c_1 \gt 0$
  \begin{align*}
          \|w(\cdot, t)\|_{L^\infty(\Omega)}
    &\le  \|\ure^{t (\Delta-1)} w_0\|_{L^\infty(\Omega)} + \int_0^t \|\ure^{(t-s) (\Delta-1)} v(\cdot, s) \|_{L^\infty(\Omega)} \ds \\
    &\le  \ure^{-t} \|w_0\|_{L^\infty(\Omega)}
          + c_1 \sup_{s \in (0, t)} \|v(\cdot, s)\|_{L^1(\Omega)} \int_0^t (1 + (t-s)^{-\frac{1}{2}}) \ure^{-(t-s)} \ds \\
    &\le  \|w_0\|_{L^\infty(\Omega)}
          + c_1 \sup_{s \in (0, \tmax)} \|v(\cdot, s)\|_{L^1(\Omega)} \int_0^\infty (1 + s^{-\frac{1}{2}}) \ure^{-s} \ds
  \end{align*}
  in $(0, \tmax)$.
  The statement then follows because of Lemma~\ref{lm:l1_bdd} and finiteness of the last integral in the inequality above.
\end{proof}

Using $L^p$-$L^q$ estimates as in Lemma~\ref{lm:bdd_nabla_v_l_q} in order to obtain $L^q$ bounds for $v_x$
does not seem to be helpful here, as an $L^1$ bound for $u$ does not imply an $L^p$ bound for $f(u)$ for any $p \ge 1$ anymore.
Therefore, we derive boundedness of a certain functional instead,
relying at a crucial point on the strength of the Gagliardo-Nirenberg inequality in the one-dimensional setting.

\begin{lemma} \label{lm:n1_bdd_vx}
  Let $n = 1$ and $\alpha \in (1, 2)$.
  Then there exists $C \gt 0$ such that
  \begin{align*}
    \|v_x(\cdot, t)\|_{L^2(\Omega)} \le C \quad \text{for all $t \in (0, \tmax)$}.
  \end{align*}
\end{lemma}
\begin{proof}
  We test the first equation in \eqref{prob:p} with $u^{\alpha - 1}$ and use Young's inequality to obtain $c_1 \gt 0$ such that
  \begin{align*}
            \frac{1}{\alpha} \ddt \intom u^\alpha
    &\le  - (\alpha - 1) \intom u^{\alpha - 2} u_x^2
          + (\alpha - 1) \intom u^{\alpha - 1} u_x v_x
          - \intom u^\alpha
          + \kappa \intom u^{\alpha - 1} \\
    &\le  - \frac{4(\alpha-1)}{\alpha^2} \intom |(u^\frac{\alpha}{2})_x|^2
          - \frac{\alpha-1}{\alpha} \intom u^\alpha v_{xx}
          - \frac12 \intom u^\alpha
          + c_1
  \end{align*}
  holds in $(0, \tmax)$.

  Lemma~\ref{lm:n1_bdd_w} provides $c_2 \gt 0$ such that $|w| \le c_2$ in $\ombar \times (0, \tmax)$.
  Hence, multiplying the second equation in \eqref{prob:p} with $-v_{xx}$ yields
  \begin{align*}
            \frac12 \ddt \intom v_x^2 
    &=    - \intom v_{xx}^2
          - \intom v_x^2
          - \intom f(u) w v_{xx} \\
    &\le  - \intom v_{xx}^2
          - \intom v_x^2
          + c_2 \intom \left( K_f u^\alpha \mathds 1_{\{u \ge 1\}} + \|f\|_{C^0([0, 1])} \mathds 1_{\{u \lt 1\}} \right) |v_{xx}|
  \end{align*}
  in $(0, \tmax)$.

  Furthermore, by again relying on Young's inequality we have 
  \begin{align*}
          -\frac{\alpha-1}{\alpha} \intom u^\alpha v_{xx}
          + c_2 \intom \left( K_f u^\alpha \mathds 1_{\{u \ge 1\}} + \|f\|_{C^0([0, 1])} \mathds 1_{\{u \lt 1\}} \right) |v_{xx}|
     \le  c_3 \intom u^{2\alpha} + \intom v_{xx}^2 + c_4
  \end{align*}
  in $(0, \tmax)$,
  where we have set $c_3 := \frac{(c_2 K_f)^2 + (\frac{\alpha-1}{\alpha})^2}{2} \gt 0$
  and $c_4 := \frac{(c_2 \|f\|_{C^0([0, 1])})^2}{2} \gt 0$.
  We conclude that the inequality
  \begin{align} \label{eq:n1_bdd_vx:func_ineq_1}
            \ddt \left( \frac{1}{\alpha} \intom u^\alpha + \frac12 \intom v_x^2 \right) 
    &\le  - \left( \frac12 \intom u^\alpha + \intom v_x^2 \right)
          - \frac{4(\alpha-1)}{\alpha^2} \intom |(u^\frac{\alpha}{2})_x|^2
          + c_3 \intom u^{2 \alpha}
          + c_5
  \end{align}
  holds in $(0, \tmax)$, where $c_5 := c_1 + c_4$.

  The Gagliardo-Nirenberg inequality provides $c_6 \gt 0$ such that
  \begin{align*}
        \|\varphi\|_{L^4(\Omega)}^4
    \le c_6 \|\varphi_x\|_{L^2(\Omega)}^{4b} \|\varphi\|_{L^\frac{2}{\alpha}}^{4(1-b)}
        + c_6 \|\varphi\|_{L^\frac{2}{\alpha}}^{4}
    \quad \text{for all $\varphi \in W^{1, 2}(\Omega)$},
  \end{align*}
  wherein
  \begin{align*}
    b = \frac12 \cdot \frac{2 \alpha - 1}{\alpha + 1} \lt \frac12,
  \end{align*}
  as $\alpha \lt 2$.

  By Lemma~\ref{lm:l1_bdd} there exists $c_7 \gt 0$ such that $\|u(\cdot, t)\|_{L^1(\Omega)} \le c_7$ for all $t \in (0, \tmax)$.
  Hence,
  \begin{align} \label{eq:n1_bdd_vx:u_2alpha_gn}
          \intom u^{2 \alpha}
    &=    \|u^\frac{\alpha}{2}\|_{L^4(\Omega)}^4 \notag \\
    &\le  c_6 \|(u^\frac{\alpha}{2})_x\|_{L^2(\Omega)}^{4b} \|u^\frac{\alpha}{2}\|_{L^\frac{2}{\alpha}}^{4(1-b)}
          + c_6 \|u^\frac{\alpha}{2}\|_{L^\frac{2}{\alpha}}^{4} \notag \\
    &=    c_6 \left( \intom (u^\frac{\alpha}{2})_x^2 \right)^{2b} c_7^{2 \alpha (1-b)} + c_7^{2 \alpha}
  \end{align}
  in $(0, \tmax)$.
  
  As $2b \lt 1$, Young's inequality implies the existence of $c_8 \gt 0$ such that
  \begin{align} \label{eq:n1_bdd_vx:u_x_young}
        \left( \intom (u^\frac{\alpha}{2})_x^2 \right)^{2a}
    \le \frac{4(\alpha-1)}{c_3 c_6 c_7^{2 \alpha (1-b)}\alpha^2} \intom (u^\frac{\alpha}{2})_x^2 + c_8
    \quad \text{in $(0, \tmax)$}.
  \end{align}
  Combining \eqref{eq:n1_bdd_vx:func_ineq_1}, \eqref{eq:n1_bdd_vx:u_2alpha_gn} and \eqref{eq:n1_bdd_vx:u_x_young} gives
  \begin{align*}
            \ddt \left( \frac{1}{\alpha} \intom u^\alpha + \frac12 \intom v_x^2 \right) 
    &\le  - c_9 \left( \frac{1}{\alpha} \intom u^\alpha + \frac12 \intom v_x^2 \right)
          + c_{10}
    \quad \text{in $(0, \tmax)$}
  \end{align*}
  for some $c_9, c_{10} \gt 0$, so that the statement follows by an ODE comparison argument.
\end{proof}

\section{$L^\infty$ bound for $u$. Proof of Theorem~\ref{th:main}}
We now use the $L^q$ bounds for $\nabla v$ to obtain
\begin{lemma} \label{lm:bdd_u_l_infty}
  If $\alpha \lt \frac{2}{n}$,
  then $\{u(\cdot, t): t \in (0, \tmax)\}$ is bounded in $L^\infty(\Omega)$.
\end{lemma}
\begin{proof}
  We follow an idea used in \cite[Lemma 3.2]{BellomoEtAlMathematicalTheoryKeller2015}.
  
  Lemma~\ref{lm:bdd_nabla_v_w} (for the case $\alpha \le 1$) and Lemma~\ref{lm:n1_bdd_vx} (for the case $\alpha \gt 1$)
  provide $q \gt n$ and $c_1 \gt 0$ such that
  \begin{align*}
    \|\nabla v(\cdot, t)\|_{L^q(\Omega)} \le c_1 \quad \text{for all $t \in (0, \tmax)$.}
  \end{align*}
  We fix an arbitrary $r \in (n, q)$.

  As
  \begin{align*}
        0
    \le u(\cdot, t)
    \le u(\cdot, t) + \int_0^t \ure^{(t-s)(\Delta -1)} f(u(\cdot, s)) w(\cdot, s) \ds
    \quad \text{for $t \in (0, \tmax)$}
  \end{align*}
  by Lemma~\ref{lm:ex_crit} and positivity of the Neumann heat semigroup,
  we have
  \begin{align*}
          \|u(\cdot, t)\|_{L^\infty(\Omega)}
    &\le  \left\|u(\cdot, t) + \int_0^t \ure^{(t-s)(\Delta -1)} f(u(\cdot, s)) w(\cdot, s) \ds\right\|_{L^\infty(\Omega)}
  \end{align*}
  for all $t \in (0, \tmax)$.

  Therefore smoothing estimates of the Neumann Laplace semigroup (cf.\ \cite[Lemma~1.3~(iv)]{WinklerAggregationVsGlobal2010})
  and the maximum principle
  yield $c_2 \gt 0$ such that
  \begin{align*}
    &\pe  \|u(\cdot, t)\|_{L^\infty(\Omega)} \\
    &\le  \|\ure^{t (\Delta-1)} u_0\|_{L^\infty}
          + \int_0^t \|\ure^{(t-s) (\Delta-1)} \nabla \cdot (u(\cdot, s) \nabla v(\cdot, s))\|_{L^\infty(\Omega)} \ds
          + \kappa \int_0^t \ure^{-(t-s)} \ds \\
    &\le  \ure^{-t} \|u_0\|_{L^\infty}
          + c_2 \int_0^t (1 + (t-s)^{-\frac12 - \frac{n}{2r}}) \ure^{-(t-s)}\|u(\cdot, s) \nabla v(\cdot, s)\|_{L^r(\Omega)} \ds
          + \kappa \\
    &\le  \|u_0\|_{L^\infty}
          + c_2 \sup_{s \in (0, t)} \|u(\cdot, s) \nabla v(\cdot, s)\|_{L^r(\Omega)}
            \int_0^\infty (1 + s^{-\frac12 - \frac{n}{2r}}) \ure^{-s} \ds
          + \kappa
  \end{align*}
  holds for all $t \in (0, \tmax)$.

  As $r \gt n$ we have therein
  \begin{align*}
    \int_0^\infty (1 + s^{-\frac12 - \frac{n}{2r}}) \ds \lt \infty
  \end{align*}
  and we may invoke Hölder's inequality to obtain
  \begin{align*}
          \sup_{s \in (0, t)} \|u(\cdot, s) \nabla v(\cdot, s)\|_{L^r(\Omega)}   
    &\le  \sup_{s \in (0, t)} \left( \|u(\cdot, s)\|_{L^\frac{rq}{q-r}} \|\nabla v(\cdot, s)\|_{L^q(\Omega)} \right) \\
    &\le  \sup_{s \in (0, t)} \left( \|u(\cdot, s)\|_{L^\infty(\Omega)}^\theta \|u(\cdot, s)\|_{L^1(\Omega)}^{1-\theta} \|\nabla v(\cdot, s)\|_{L^q(\Omega)} \right)
  \end{align*}
  for all $t \in (0, \tmax)$ and $\theta := 1 - \frac{q-r}{rq} \in (0, 1)$.

  Therefore, by Lemma~\ref{lm:l1_bdd} the function
  \begin{align*}
    L: (0, \tmax) \ra \R, \quad t \mapsto \sup_{s \in (0, t)} \|u(\cdot, s)\|_{L^\infty(\Omega)}
  \end{align*}
  fulfills
  \begin{align*}
    L(t) \le c_3 \left(1 +  L(t)^\theta \right) \quad \text{for all $t \in (0, \tmax)$}
  \end{align*}
  for some $c_3 \gt 0$.

  However, as $\theta \in (0, 1)$, this implies
  \begin{align*}
        \sup_{t \in (0, \tmax)} \|u(\cdot, t)\|_{L^\infty(\Omega)}
    &=  \sup_{t \in (0, \tmax)} L(t)
    \le C
  \end{align*}
  for some $C \gt 0$.
\end{proof}

Equipped with these estimates we are now able to prove our main theorem.
\begin{proof}[Proof of Theorem~\ref{th:main}]
  Existence of a uniquely determined solution as well as nonnegativity is guaranteed by Lemma~\ref{lm:ex_crit}.
  The extensibility criterion in Lemma~\ref{lm:ex_crit} in combination
  with Lemma~\ref{lm:bdd_nabla_v_w}, Lemma~\ref{lm:n1_bdd_w}, Lemma~\ref{lm:n1_bdd_vx} and Lemma~\ref{lm:bdd_u_l_infty} then implies
  $\tmax = \infty$ as well as $|u|, |w| \le C$ in $\Omega \times (0, \infty)$ for some $C \gt 0$.

  As then $f(u) w \le M$ in $\Omega \times (0, \infty)$ for some $M \gt 0$
  another (straightforward) application of smoothing estimates of the Neumann Laplace semigroup
  finally also yields boundedness of the set $\{(v(\cdot, t): t \in (0, \infty)\}$ in $W^{1, \infty}(\Omega)$.
\end{proof}

\sloppy \printbibliography
\end{document}